\documentclass[reqno,11pt]{amsart}
\usepackage{amsmath,amssymb,color}

\newtheorem{theorem}{Theorem}[section]
\newtheorem{proposition}[theorem]{Proposition}

\newtheorem{lemma}[theorem]{Lemma}
\newtheorem{corollary}[theorem]{Corollary}

\newtheorem{remark}[theorem]{Remark}

\numberwithin{equation}{section}
\numberwithin{theorem}{section}
\renewcommand{\phi}{\varphi}

\newcommand{\eps}{\varepsilon}

\newcommand{\nada}[1]{}

\newcommand{\R}{\mathbb R}
\renewcommand{\O}{\mathcal O}
\renewcommand{\S}{\mathbb S}

\newfont{\indic}{bbmss12}

\definecolor{light}{gray}{.97}

\title{On the regularity of timelike extremal surfaces}

\author{R. L. Jerrard \and M. Novaga \and G. Orlandi}
\address{Department of Mathematics, University of Toronto, Toronto, Ontario, Canada}\email{rjerrard@math.toronto.edu}
\address{Dipartimento di Matematica, Universit\`a di Padova, Padova, Italy}\email{novaga@math.unipd.it}
\address{Dipartimento di Informatica, Universit\`a di Verona, Verona, Italy}\email{giandomenico.orlandi@univr.it}

\date{}

\begin{document}

\begin{abstract}
We study  a class of timelike weakly  extremal %$1+1$-dimensional
surfaces in flat Minkowski space $\R^{1+n}$,
characterized by the
fact that they admit a $C^1$  parametrization 
(in general not an immersion)
of a specific form. We prove that if the distinguished
parametrization is of class $C^k$, then the surface is regularly immersed 
away from a closed singular 
set of euclidean Hausdorff dimension at most $1+1/k$, and that 
this bound is sharp. 
We also show that, generically with respect
to a natural topology, 
the singular set
of a timelike weakly extremal cylinder in $\R^{1+n}$ is $1$-dimensional
if $n=2$, and it is empty 
if $n \ge 4$.
For $n=3$, timelike weakly extremal surfaces exhibit an intermediate behavior.
%The proof is base on the use of the so-called orthogonal gauge for parametrizing timelike minimal cylinders. 
\end{abstract}

\maketitle

\section{Introduction}

In this paper we study timelike extremal surfaces in 
$(1+n)$-dimensional flat Minkowski space. In particular, we focus on extremal
immersions of a cylinder $\R\times\S^1$ into $\R^{1+n}$, which arise in models of 
closed cosmic strings and have been extensively studied in the physics
community (see \cite{BoIn:34,VS,An:03} and references therein), as well as
in the more recent mathematical literature
\cite{Neu,Li:04,Mi:08,KZZ,KZ,BHNO,J,NT}, and have recently
been proved \cite{BNO1,J} to describe the dynamics of topological defects in various
relativistic field theories in certain scaling limits.

As for many geometric problems, timelike extremal surfaces present various
kinds of singularities. For instance, it has been shown in \cite{BHNO} that a
closed convex string in $\R^2$ with zero initial velocity shrinks to a point
in finite time, while its shape approaches that of a circle. An analogous phenomenon
can be found in other geometric evolutions such as the planar curvature flow
\cite{GH} and the hyperbolic curvature flow of convex curves \cite{KKW}.
However more complicated singularities can occur during the evolution
(typically the formation of cusps), and a partial classification has been
provided in \cite{EH}, where the authors study self-similar singularity
formation.
A theory of generalized extremal surfaces in the varifolds sense (see
\cite{allard}) has been recently proposed in \cite{BNO1,BNO2}.

In \cite{NT} it has been shown that any (immersed) timelike extremal cylinder 
in $\R^{1+2}$ necessarily develops singularities in finite time. 
%, so that in particular there are no globally smooth timelike minimal surfaces with compact space slices. 
In the same paper, the authors conjecture that this does not hold in $\R^{1+n}$
for $n\ge 3$, where existence of smooth timelike extremal cylinders is expected.
On the other hand, there exist globally smooth timelike extremal
surfaces with noncompact slices in $\R^{1+2}$, which are small perturbations of timelike
planes \cite{Li:04, KZZ}. % (see also \cite{Li:04, Br:02} for similar results concerning higher-dimensional hypersurfaces).

The arguments of \cite{NT} rely heavily on a particular representation of 
extremal immersed cylinders, which we call the
{\em orthogonal gauge},
known for a long time in the physics literature and first
proved to be valid, as far as we know, in \cite{BHNO}. 
This representation also yields global weak solutions
in the sense of \cite{BNO1, BNO2}. 
The main goal of this paper is
to estimate the dimension of the singular set of these 
weak solutions, which have the good property that they
are images of $C^1$ maps (in general {\em not immersions})
of a specific form, see \eqref{eqrepr}, \eqref{vincoloab} below.
In particular we prove that, if the map is of class
$C^k$, then the dimension of the singular set 
is bounded above by $1+\frac 1 k$,
and the bound
is sharp.
%
%In fact, it follows directly from classical results that 
%{\em any} surface parametrized by a $C^k$ map can have a 
%singular set (in the sense that we consider) of dimension at most $1+\frac 1k$.
The upper bound on the dimension turns out to follow immediately
from a classical refinement of Sard's Theorem, due to Federer \cite{federer}, so
the construction of examples of {extremal} surfaces attaining this bound is the
harder part of this result.

We also show that the singular set is generically empty when $n>3$, 
confirming the conjecture of \cite{NT} in such dimensions. More precisely we
show that, generically for $n>3$, given a closed curve $\Gamma$ immersed in
$\R^n$ and a velocity field $v:\Gamma\to \R^n$, with $|v|<1$ and orthogonal
to $\Gamma$, there exists a smooth globally immersed timelike extremal
surface containing $\Gamma$ and tangent to $(1,v)$.
For $n=3$, roughly speaking, both globally smooth immersed solutions
and solutions that develop singularities occur for
large sets of initial data (that is, sets with nonempty interior.)

We start in Section \ref{S:2} by quickly recalling some properties of the orthogonal
gauge, including existence and (restricted) uniqueness of solutions
of a Cauchy problem for timelike extremal surfaces.
We also present some examples in Section \ref{S:4} showing that 
uniqueness may fail without the restrictions imposed in Section \ref{S:2}.

\section{Timelike extremal surfaces in the orthogonal gauge}\label{S:2}

Given an open interval $I\subset\R$,
and an immersion 
$\psi: I\times \R \to \R^{1+n}$, possibly periodic with respect to the second variable, 
for an open set $U\subset I\times \R$ we define the Minkoswkian area of
$\psi(U)$  to be
\[
\int_{U} \sqrt{| g|}\ ,
\qquad
g := \det (g_{ij}),
\qquad g_{ij} := (\partial_i \psi, \partial_j\psi)_m 
\]
where $(\cdot, \cdot)_m$ denotes the Minkowski inner product. 
This functional is also sometimes called the Nambu-Goto action.
%and $a,b$ run from $0$ to $1$. 
The surface parametrized by $\psi$
is said to be {\em timelike} if $g<0$ everywhere, and a timelike surface
is {\em extremal} if $\psi$ is a critical point of the Minkowskian area
functional with respect to compactly supported variations.

It is noted in \cite{NT} that any timelike immersion of a surface into $\R^{1+n}$
can be reparametrized locally to have the form  
\begin{equation}
\psi(t,x) = (t, \gamma(t,x)). %,\qquad\quad \gamma:U \to \R^n.
\label{formofpsi}\end{equation}
Here we will consider the initial value problem for timelike extremal surfaces
with initial data of the form
\begin{equation}
\gamma(0,x) = \gamma_0, \quad
\gamma_t(0,x) = v_0,
\label{initconduno}\end{equation}
where $\gamma_0\in C^1(\R ; \R^n)$ is an
immersion and $v_0\in C^0(\R;\R^n)$ satisfies $v_0\cdot \gamma_{0}' = 0$ and $|v_0|<1$ everywhere. 
We  call such a pair an {\em admissible couple}, and we say that an admissible couple is {\it periodic} if 
%$\gamma_0$ parametrizes a closed curve,  that is, 
$\gamma_0$ and $v_0$ are periodic with the same period, which implies
in particular that $\gamma_0$ parametrizes a closed curve.
We remark that if $\gamma_0$ is an embedding, or more generally
if $v_0\circ \gamma_0^{-1}$ is single-valued 
on $\mbox{Image}(\gamma_0)$, then the initial condition \eqref{initconduno}
can be restated in the form
\begin{equation}
\mbox{$\gamma_0$ parametrizes $\{x \in \R^ n : (0,x)\in M \}$,\ and $(1, v_0(x))\in T_{\psi(0,x)}M$ for every $x\in \R$.}
\label{geometric.ic}\end{equation}
Two admissible couples  $(\gamma_0,v_0)$, $(\hat \gamma_0, \hat v_0)$ are considered to be equivalent if there is a $C^1$ diffeomorphism 
$\lambda:\R\to \R$ such that $(\gamma_0,v_0) =
(\hat \gamma_0, \hat v_0)\circ \lambda $. Equivalent 
couples encode exactly the same geometric data, and 
to any timelike surface $M$, whose $t=0$ slice is an immersed curve,
one can assign an (equivalence class of) admissible couples,
indeed possibly multiple equivalence classes if the curve is
not embedded.

Our approach is  based on the observation, classical in the physics literature and straightforward
to verify (see \cite{VS,BHNO}),
that if $\gamma\in C^k( I\times \R; \R^{n})$, $k\ge 1$ satisfies
\begin{eqnarray}\label{vincolozero}
\label{vincolouno}
|\gamma_x|^2 - |\gamma_t|^2&=&1
\\\label{vincolodue}
\gamma_x\cdot\gamma_t&=&0
\\\label{onde}
\gamma_{tt} - \gamma_{xx} &=& 0 %\qquad \textrm{(in the distributional sense)}
\end{eqnarray}
for all $(t,x)\in I\times \R$, then $\psi(x,t) = (t, \gamma(t,x))$
is a solution of the Euler-Lagrange equations associated to the Minkowski
area functional wherever $g\ne 0$, and hence is an extremal immersion near such points.
This holds in 
the distributional sense if $k=1$ and classically if $k\ge 2$.
In view of \eqref{vincolouno}-\eqref{vincolodue},
we will call such a parametrization the {\it orthogonal gauge}.
%It is also sometimes called a {\em conformal parametrization}, since \eqref{vincolouno}-\eqref{vincolodue} state that the metric tensor $(g_{ab})$ is conformally equivalent to the standard $1+1$-dimensional Minkowski metric.

The general solution $\gamma$ of
\eqref{vincolouno} - \eqref{onde} has the form
\begin{equation}\label{eqrepr}
\gamma(t,x)=\frac{a(x+t)+b(x-t)}{2}
\end{equation}
where $a,b\in C^1(\R;\R^n)$ are maps satisfying 
\begin{equation}\label{vincoloab}
|a'|=|b'|=1 \quad {\rm in\ }\R .
\end{equation}
Indeed, \eqref{eqrepr} is just d'Alembert's formula, and once $\gamma$ is known
to have the form \eqref{eqrepr}, then the constraints \eqref{vincolouno}, \eqref{vincolodue}
are easily seen to be equivalent to \eqref{vincoloab}.

Given a function $\gamma(t,x) = \frac 12( a(x+t)+b(x-t))$, with $a,b$ satisfying \eqref{vincoloab},
we shall write in the sequel $\psi(t,x) := (t,\gamma(t,x))$ and $M := \mbox{Image}(\psi)$.
%In this section we prove some results giving bounds on the size of certain singular sets
%that can be associated to  $M$.
We also define the singular set of $M$ as
\[
Sing := \{ \psi(t,x) : \mbox{rank}(\nabla \psi)(t,x) < 2 \}  \ = \ \{ \psi(t,x) : \gamma_x(t,x) = 0\}.
\]
We have that $M$ is timelike and regularly immersed in an open neighborhood of every  point of $M\setminus Sing$, 
while, at every point of $Sing$, the orthogonal coordinate system degenerates and,
as we will prove in Theorem \ref{teoex}
below, $M$ fails to be timelike. A stricter notion of singular set is
\[
Sing^* := \{ p\in Sing : \lim_{q\in M, q\to p} \tau(q)\mbox{ does not exist} \},
\]
where $\tau(\cdot)$ is the (spatial) tangent
%\[
%Reg := M \setminus Sing.
%\]
%We think of $Reg$ as the set of timelike, regularly immersed points of $M$. 
%At points in $Sing$, $M$ certainly fails to be timelike, but it could otherwise be rather smooth.
\[
\tau(p) =  \frac{\gamma_x}{|\gamma_x|}\circ \psi^{-1}(p)
\]
defined wherever it makes sense, which is at points $p\in M\setminus Sing$ where
the set $\{\frac{\gamma_x}{|\gamma_x|}(t,x) :
\psi(t,x)=p\}$ consists of exactly one element. %It is possible that there is someother better notion of a strongly singular set in this context. 

We note that the definitions of $Sing$ and  $Sing^*$  both have the drawback that they depend on
the parametrization of $M$.

We collect some known results in the following

\begin{proposition}\label{prop:summary}
Given an admissible couple $(\hat \gamma_0, \hat v_0)\in C^k\times C^{k-1}$, there exists 
an equivalent admissible couple $(\gamma_0, v_0)$ and a 
map $\gamma\in C^k(\R\times \R; \R^n)$ of the form \eqref{eqrepr}, \eqref{vincoloab}, such that the initial condition \eqref{initconduno}
holds. In addition,

{\bf 1}. 
$\psi(t,x) = (t, \gamma(t,x))$ is timelike and an immersion in a neighborhood of every point where
$\gamma_x\ne 0$, and it is neither timelike nor an immersion at points where $\gamma_x$ vanishes.

{\bf 2}. 
$\psi$ is an extremal immersion wherever it is a immersion, 
and in particular this holds for $(t,x)$ in a neighborhood of $\{0\}\times \R$.

{\bf 3}. 
If $\hat \psi$ is any extremal immersion of the form
$\hat\psi(t,x) = (t, \hat\gamma(t,x))$ for $(t,x)\in I\times \R$ for some interval $I\subset \R$
containing $0$, and if
$(\hat \gamma(0,\cdot), \hat \gamma_t(0,\cdot))$ is equivalent to $(\gamma_0, v_0)$, then
$\psi$ is a reparametrization of $\hat \psi$, and thus $\psi(I\times \R) 
= \hat \psi(I\times \R)$.

{\bf 4}. 
$M = \mbox{Image}(\psi)$ can be identified with a global
weak solution of the extremal
surface equation, in the sense of \cite{BNO1,BNO2}.
\end{proposition}

This proposition implies in particular the local existence of a smooth 
timelike extremal surface
$M$ satisfying the initial condition \eqref{geometric.ic} for an admissible couple
$(\gamma_0, v_0)$ such that $v_0\circ \gamma_0^{-1}$ is single-valued, as well as 
the global existence of a weak solution.

We show in Proposition \ref{prop:nonuniq} below that the restriction
of the uniqueness assertion {\bf 3} to the class of surfaces parametrized by
maps to the form \eqref{formofpsi} is in fact necessary; without this condition,
uniqueness can fail.

\begin{proof}
Given any admissible couple $(\hat \gamma_0, \hat v_0)$, we can always find
an equivalent couple $(\gamma_0, v_0)$ such that
\begin{equation}\label{normalize.data}
|\gamma_{0}'|^2 + |v_0|^2 = 1.
\end{equation}
Letting $\gamma$ denote the solution of the wave equation \eqref{onde} with
initial data  \eqref{initconduno},
it is easy to check (see for example \eqref{eq:aprime} and \eqref{eq:bprime} below) that 
$\gamma$ satisfies \eqref{eqrepr}, \eqref{vincoloab}, thus proving the existence of
an extremal immersion for the admissible couple $(\gamma_0, v_0)$.

The proof of  {\bf 1} is given in the proof of Theorem \ref{teoex} below.
The only subtle part is checking that $M$ is not timelike at $\psi(t,x)$,
if $\gamma_x(t,x) = 0$; everything else follows easily from the definitions
and \eqref{eqrepr}.

Concerning {\bf 2},
we have already noted that a straightforward
computation shows that $\psi$ is an extremal immersion wherever it is an immersion,
and it follows from {\bf 1} that $\psi$ is an immersion in a neighborhood of $\{0\}\times \R^n$.

Finally, conclusions {\bf 3} and {\bf 4} are established in \cite{BHNO} and \cite{BNO1, BNO2}
respectively. They are proved for $\gamma$ which is periodic in the $x$ variable, 
but both facts are essentially local (due to finite propagation speed) and so the proofs work without change 
in the general case.
\end{proof}

\begin{remark}
\rm In [15, Theorem 4.1], global existence of $C^2$ solutions 
is proved for an equation that, like (2-4)-(2-6), 
is equivalent to the equation for timelike extremal surfaces as long as
the surfaces associated to the solutions remain immersed. 
In this result, the orthogonal gauge is not imposed,
and the equations considered are thus nonlinear.
\end{remark}

We record  some standard formulas.
Differentiating \eqref{eqrepr} we obtain
\begin{eqnarray}\label{equno}
\gamma_x(t,x)&=&\frac{a'(x+t)+b'(x-t)}{2}
\\\label{eqdue}
\gamma_t(t,x)&=&\frac{a'(x+t)-b'(x-t)}{2}.
\end{eqnarray}
Letting $t=0$ in \eqref{equno} - \eqref{eqdue} and recalling \eqref{initconduno}, we deduce that
%\begin{eqnarray}\label{cuno}
%\gamma_x(0,x)&=&\frac{a'(x)+b'(x)}{2}=\gamma_{0}'(x)
%\\\label{cdue}
%\gamma_t(0,x)&=&\frac{a'(x)-b'(x)}{2}= v_0(x)
%\end{eqnarray}
%which gives
\begin{eqnarray}
a'(x)&=&
\gamma_{0}'(x) + v_0(x)
\label{eq:aprime}\\
b'(x)&=&
\gamma_{0}'(x) - v_0(x)
\label{eq:bprime}
\end{eqnarray}

%An admissible couple $(\hat \gamma_0, \hat v_0)$ is {\em periodic} if there exists some $E>0$ such that both $\hat \gamma_0$ and $\hat v_0$ are periodic with period $E$. 

We define a {\em cylinder} to be a set $M\subset I\times\R^n$ that can
be written  {\em globally} as the image of a map $\psi$
of the form \eqref{formofpsi}, 
where $\gamma(t, \cdot)$ is periodic with fixed period $E$ for every $t\in I$.

It is straightforward to check that if one starts with a representative
$(\hat \gamma_0, \hat v_0)$ of a periodic admissible couple such
that \eqref{normalize.data} does not hold,
with $(\hat \gamma_0, \hat v_0)$ periodic of period $L$,
then an equivalent couple  $(\gamma_0, v_0)$ that satisfies \eqref{normalize.data} 
is periodic  with period
\[
E_0  := \int_0^L \frac{|\hat \gamma_{0}'(x)|}{\sqrt{1-|\hat v_0(x))|^2}}\,dx.
\]
Then  $a+b$
is periodic, and  we see from \eqref{eq:aprime}, \eqref{eq:bprime} that 
$a', b'$ are periodic as well, all with period $E_0$.
Hence, if $(\hat \gamma_0, \hat v_0)$ is a periodic admissible
couple, then the surface associated to $(\hat \gamma_0, \hat v_0)$
by Proposition \ref{prop:summary} is a cylinder.

%Given $z\in \Gamma$ we let  $\tau(z)$ be the unit tanget vector to $\Gamma$ at $z$. 
%As $\tau(z)$ is defined up to a sign change,   from now on we fix a continuous choice of the vector field $\tau:\Gamma\to \R^n$.

%Note that \eqref{vincolouno}, \eqref{vincolodue} state that $g_{01}=g_{10}=0$ and $g_{11} = -g_{00} = \sqrt{|g|}$, or equivalently that the induced metric $(g_{ab})$ is conformally equivalent to the $1+1$-dimensional Minkowski metric.

Notice that, given a solution $\gamma$, the corresponding couple $(a,b)$ is uniquely determined 
up to additive constants. In particular, the othogonal gauge provides a
one-to-one correspondence between the set of all equivalence classes of admissible couples and the set
\[ %\begin{eqnarray*}
X \ := \ \big\{(a,b)\in C^1(\R;\R^n)\times C^1(\R;\R^n):\ 
a'+b'\ \mbox{never vanishes},  \ |a'| = |b'| =1 \big\}/\sim
\] %\end{eqnarray*}
where $(a,b)\sim (c,d)$ iff there exist $x_0\in \R$, $z_0\in\R^n$ and $\sigma_0\in \{\pm1\}$ such that 
\[
c(x)=a(\sigma_0 x+x_0)+z_0\qquad d(x)=b(\sigma_0 x+x_0)-z_0\qquad {\rm for\ all\ }x\in\R.
\]
%or
%\[
%c(x)=a(-x+x_0)+z_0\qquad d(x)=b(-x+x_0)-z_0\qquad {\rm for\ all\ }x\in\R.
%\]
Similarly, equivalence classes of periodic admissible couples are parametrized by
\[
X_{\rm per} 
\ = \  \big\{[(a,b)]\in X :\ 
\textrm{$a', b', a+b$ periodic with the same period} \big\}
\]
where $[\cdot ]$ denotes an equivalence class. 
When $(a,b)\in X_{\rm per}$, we shall denote by $E_0$ the common period of $a,\,b$.

We shall consider the topology induced by $C^1(\R;\R^n)\times C^1(\R;\R^n)$ on $X$ (or equivalently on the set of admissible couples)
and we refer to it as the $X$-topology.
We say that a property holds {\it generically} if it holds for all admissible couples
out of a closed set with empty interior with respect to this topology.

\section{Generic regularity}\label{sec:3}

In this section we study the regularity properties of extremal surfaces, which hold generically 
with respect to the $X$-topology. We start with a general regularity result
which follows directly from the orthogonal gauge parametrization.

\begin{theorem}\label{teoex}
Given an admissible couple $(\gamma_0,v_0)$, 
there exists a global timelike extremal surface $M$ of the form \eqref{formofpsi}, 
containing $\Gamma_0={\rm Image}(\gamma_0)$ and tangent to $(1,v_0)$, 
if and only if %$\Sigma=\emptyset$, that is
\begin{equation}\label{condab}
a'(s)\ne -b'(\sigma)\qquad {\rm for\ all\ }s,\sigma\in \R.
\end{equation}
If $(a,b)\in X_{\rm per}$ then $M$ is an extremal cylinder.
%or equivalently iff
%\begin{equation}\label{condz}
%\sqrt{1-|v(z_1)|^2}\,\tau(z_1) + v(z_1)
%\ne 
%-\sqrt{1-|v(z_2)|^2}\,\tau(z_2) + v(z_2) 
%\qquad {\rm for\ all\ }z_1,z_2\in \Gamma.
%\end{equation}
\end{theorem}

We have only defined {\em timelike} for immersed surfaces.
A surface $M$ given as the image of a map
$\psi$ may be smooth even where $\psi$ is not an immersion.
In this case, we say that $M$ is timelike at a point 
$p\in M$ if $T_pM$ exists and is timelike, and in addition the 
spatial unit tangent $\tau$ is continuous at $p$.

\begin{proof}
Assume \eqref{condab}. Then it is clear from the form \eqref{eqrepr} of $\gamma$ that 
$\gamma_x$ never vanishes, and from the form \eqref{formofpsi}
of $\psi$, it follows that
that $Sing = \emptyset$ and hence that $\psi$ is a global immersion. 
It follows from \eqref{vincolouno} that $|\gamma_t|<1$
whenever $\gamma_x\ne 0$, and from this it is easy to check
that $\psi$ is a timelike immersion everywhere.

If \eqref{condab} fails, then
$\gamma_x(t,x)=0$ for some $(t,x)\in\R\times\R$, and by \eqref{vincolouno} we have 
$|\gamma_t(t,x)|=1$.
We will show that $M$ is not timelike at $\psi(t,x)$.
This is clearly the case  if $p\in Sing^*$,
so we assume that $p\not \in Sing^*$. 
Then we can define a spatial tangent $\tau(p)$,
and $T_pM$ is spanned by $(0,\tau(p))$ and $( 1, \gamma_t(t,x))$. 
Thus it  suffices to show that
\begin{equation} \tau(p)\cdot \gamma_t(t,x) = 0,
\label{null}\end{equation}
since then it is easy to check that $T_pM$ contains no timelike vectors.

To prove \eqref{null}, 
fix a sequence $(t_k,x_k)$ in $M\setminus Sing$
such that $p_k := \gamma(t_k, x_k)\to \gamma(t,x)$.
(We prove in Theorem \ref{thm:sing} below that 
$\mathcal H^2(Sing) = 0$, so such a sequence exists.)
Then since $\tau$ is continuous at $p$,
\begin{equation}\label{null2}
\tau(p) := \lim_k \frac {\gamma_x(t_k,x_k)}{|\gamma_x(t_k,x_k)|},
\qquad
\gamma_t = \lim_k \gamma_t (t_k,x_k).
\end{equation}
We write $\gamma(t,x)  = \frac 12(a(x+t)+b(x-t))$ as usual, 
and we use the notation
\[
m_k := a'(x_k+t_k) , \qquad n_k :=  -b'(x_k-t_k).
\]
If we define $n_0 = a'(x+t)$ then, 
using the \eqref{null2} and the fact that 
$\gamma_x(t,x)= 0$, we find
that 
\begin{equation}\label{null3}
m_k\mbox{ and }n_k \to n_0, %\qquad n_k\to n_0 
\qquad\mbox{ as }k\to \infty,\qquad\qquad\mbox{ and \  }
n_0 = \gamma_t (t,x).
\end{equation}
Then $\gamma_x(t_k,x_k) = m_k-n_k$ and $n_0 = \gamma_t(t,x)$, so \eqref{null} reduces to showing that if \eqref{null3} holds and
$|n_k|=|m_k|=1$ for all $k$, then
\[
%\frac
|(m_k - n_k)\cdot n_0 |= o(|n_k - m_k|)
 \qquad\mbox{ as
$k\to \infty$.} %, when \eqref{null3} holds}.
\]
Writing $\theta_k := \cos^{-1} ( m_k\cdot n_0)$ and $\phi_k := \cos^{-1} ( n_k\cdot n_0)$,
it is not hard to see that $|n_k - m_k| \ge |\sin \theta_k - \sin \phi_k| \ge  \frac 12|\theta_k - \phi_k|$
for $k$ sufficiently large, and then it suffices to check that
\[
|\cos \theta_k - \cos \phi_k| = o(|\theta_k - \phi_k|)
\]
for $\theta_k,\phi_k\to 0$, which is clear.

%Finally, by \eqref{cuno} and \eqref{equnobis} the condition $\gamma_x(t,x)\ne 0$ for all $(t,x)\in\R\times\R$ is equivalent to \eqref{condab} and \eqref{condz}.
\end{proof}

Notice that condition \eqref{condab} is equivalent to say that the two curves 
$a',-b':\R\to \S^{n-1}$ do not intersect. 

The following result has been proved in \cite{NT}.

\begin{corollary}\label{coruno}
Let $n=2$ and let $(\gamma_0,v_0)$ be a periodic admissible couple. 
Then the curve $\Gamma_0={\rm Image}(\gamma_0)$ cannot be immersed in a global timelike extremal cylinder tangent to $(1,v_0)$.
\end{corollary}

\begin{remark}\label{localcyl}\rm
We emphasize that the corollary applies only to extremal {\em cylinders}.
The proof does not rule out the possibility of smooth timelike
extremal surfaces in $\R^{1+2}$ that are locally (but not
globally) cylindrical, see Proposition \ref{prop:nonuniq} below.
\end{remark}

\begin{proof}
By Theorem \ref{teoex} it is enough to show that there exist $s,\sigma\in[0,E_0]$ such that 
\[
a'(s)+b'(\sigma)=0.
\]
As $|a'|=|b'|=1$ and 
\begin{equation}\label{eqab}
\int_0^{E_0}a'(s)\,ds=\int_0^{E_0}-b'(\sigma)\,d\sigma\,, 
\end{equation}
the supports of the curves $a'$ and $-b'$ are two connected arcs of $\S^1$, which necessarily intersect.
%with angular span strictly greater than $\pi$.
The thesis then follows from Theorem \ref{teoex}.
\end{proof}

\begin{corollary}\label{cordue}
Let $n=3$ and let $(\gamma_0,v_0)$ be a periodic admissible couple. 
If $\Gamma_0=\mbox{Image}(\gamma_0)$ can be immersed in a global timelike extremal cylinder tangent to $(1,v_0)$,
then the same holds for any periodic admissible couple $(\hat\gamma_0,\hat v_0)$,
sufficiently close to $(\gamma_0,v_0)$ in the $X$-topology.

Conversely, if $\Gamma_0$ cannot be immersed in a global timelike extremal cylinder tangent to $(1,v_0)$,
then generically the same holds for any couple $(\hat\gamma_0,\hat v_0)$ sufficiently close to $(\gamma_0,v_0)$.
\end{corollary}

\begin{proof}
The first assertion follows immediately from the fact that the set of couples $(a,b)$ satisfying \eqref{condab}
is open in $X_{\rm per}$.

The second assertion follows by noticing that 
the set of couples $(a,b)$ such that the support of $a'$ intersects at least two connected components
of the complement in $\mathbb S^2$ of the support of $-b'$ 
%curves $a',-b':[0,E_0]\to \S^2$ intersect transversally
is an open set in $X_{\rm per}$, 
while the set of couples $(a,b)$ such that 
the support of $a'$ intersects only one connected component
of the complementary of the support of $-b'$ is a closed set with empty interior.
\end{proof}

\begin{remark}\rm
If we consider data $(\gamma_0,v_0)$ parametrized by $X^2 := X_{\rm per}\cap \left(C^2(\R)\times C^2(\R)\right)$, 
endowed with the stronger topology induced by $C^2(\R)\times C^2(\R)$,
then $\Gamma_0$ can generically be immersed in a global $E_0$-periodic surface tangent to $(1,v_0)$,
which is a timelike extremal surface away from a discrete set of singular points, parametrized by the finite set $Sing$.
Moreover, the cardinality of the singular set $Sing$
is invariant for small perturbations of $(\gamma_0,v_0)$ in the $X^2$-topology. 
Indeed, we observe that the couples $(a,b)$ such that 
the curves $a'$ and $-b'$ have a finite number of transversal intersections is a dense open set in $X_{\rm per}$
with respect to the $X^2$-topology, and the number of intersections is locally constant.
Hence the curve $\gamma$ given by
\eqref{eqrepr} parametrizes a $E_0$-periodic timelike extremal cylinder tangent to $(1,v_0)$, 
away from a singular set which is finite in $[0,E_0]\times \R^3$, and 
the number of singularities is invariant for small perturbations of $(\gamma_0,v_0)$.
\end{remark}

An example of admissible couple in $\R^3$ which is immersed in a global timelike extremal cylinder
has been given in \cite{NT}. More generally, we prove in
Lemma \ref{lem:convexhull} below that any curve in $\S^{2}$ whose convex hull contains 
a neighborhood of the origin can be realized as the set of tangent vectors
of a closed curve $a$ such that $|a'| = 1$. Hence one can easily find pairs
$a,b:\R\to \R^n, n\ge 3$
of periodic curves with the same period, such that $a'$ and $-b'$ trace out disjoint curves in
$\mathbb S^{n-1}$. By Theorem \ref{teoex}, each such pair
yields an example of a globally smooth timelike extremal cylinder. 

\begin{lemma}
Assume that $c: \S^1\to \S^{n-1}$ is a smooth closed curve such
that $0$ belongs to co$($Image$(c))$, where co$(\cdot)$ denotes the convex hull.
Then there exists a closed curve $a:\S^1\to \R^n$, of the
same smoothness as $c$, such that $\mbox{Image}\,(a') = \mbox{Image}\,(c)$.
\label{lem:convexhull}
\end{lemma}

\begin{proof}
We write $c$ as a $2\pi$-periodic function from $\R$ to $\S^{n-1}$.
By assumption there exist points $0< x_0 < \ldots < x_{n} \le 2\pi$
such that 
\begin{equation}
0 \in \mbox{int} \left(\mbox{co} \{c(x_0), \ldots, c(x_n) \}\right)
\label{eq:intco}\end{equation}
Let $p:\R\to \R$ be a smooth increasing function such that 
$p(x+2\pi) = p(x)+2\pi$, 
\[
p(x_i) = x_i, \qquad \mbox{ and }\quad
\frac {d^k p}{dx^k}(x_i) = 0 \mbox{ for every $k\in N$ and  $i=0,\ldots, n$}.
\]
Then, given positive numbers $\ell_0,\ldots, \ell_n$, let 
$L_i := \sum_{j=0}^i \ell_j$ and define
\[
\tilde c(x) := 
\begin{cases}
c(p(x))&\mbox{ for }0\le x \le x_0\\
c(x_0)&\mbox{ for } x_0 \le x \le x_0+L_0 \\
c(p(x - L_0))&\mbox{ for }x_0+L_0 \le x \le x_1+L_0\\
%c(x_1) &\mbox{ for }x_1 + L_0 \le x \le x_1 + L_1\\
\quad\vdots&\qquad\qquad\vdots \\
c(x_n)&\mbox{ for }x_n + L_{n-1} \le x \le x_n + L_n\\
c(p(x - L_n))&\mbox{ for }x_n+L_n \le x \le 2\pi + L_n.
\end{cases}
\]
We claim that one can choose positive $(\ell_i)$ so that $\int_0^{2\pi + L_n} \tilde c(x) \ dx = 0$.
Indeed, since
\[
\int_0^{2\pi + L_n} \tilde c(x) \ dx = 0 \ = \ 
\int_0^{2\pi} c(p(x)) dx + \sum_{i=0}^n \ell_i c(x_i),
\]
the claim follows from \eqref{eq:intco}.

We now fix $(\ell_i)$ as above and define $\hat c(x) := \tilde c\left( \frac {(2\pi + L_n) x}{2\pi}\right)$. Then $a(x) := \int_0^x \hat c(y) dy$,
for $0\le x \le 2\pi$, defines a closed curve with the the required properties.
\end{proof}

\begin{corollary}\label{cortre}
Let $n>3$ and let $(\gamma_0,v_0)$ be a periodic admissible couple. 
Then $\Gamma_0$ can be generically immersed in a global timelike extremal cylinder tangent to $(1,v_0)$.
\end{corollary}

\begin{proof}
The assertion follows as before from the fact that the set of couples $(a,b)$ satisfying \eqref{condab}
is open in $X_{\rm per}$, while the set of couples $(a,b)$ such that the curves $a',-b':[0,E_0]\to \S^2$ intersect 
is a closed set with empty interior.
\end{proof}

%\begin{remark}\rm
%We don't know if the results in Corollaries \ref{cordue} and \ref{cortre} still hold if
%we consider on $X$ the (weaker) topology induced by $C^0(\R)\times C^0(\R)$. 
%\end{remark}

\begin{remark}\rm
A related question is what happens if we assume $\gamma_0$ to be an embedded curve in $\R^n$ and ask if it is contained 
in a global {\it embedded} timelike extremal surface in $\R^{1+n}$. It is easy to check that 
Corollary \ref{coruno} still holds in this case, and we expect that  
Corollaries \ref{cordue} and \ref{cortre} also hold, with similar proofs.
\end{remark}

\begin{remark}\rm 
As the set of periodic admissible couples which can be immersed in a global timelike extremal cylinder
%$(a,b)\in X_{\rm per}$ 
is parametrized by an open subset $\O\subset X_{\rm per}$, 
it is natural to speculate on the number of connected components of $\O$.
While it is clear from Corollary \ref{coruno} that $\O=\emptyset$ if $n=2$, it is not difficult to show that
$\O$ has infinitely many connected components if $n=3,4$, while $\O$ is connected if $n>4$.

Indeed, if $n=3$ and 
$a',-b'$ are two disjoint closed curves in $\S^2$, 
%(even if embeddedness is not necessarily guaranteed),
then the winding number of $a'$ around the image of $-b'$ is constant on connected components of $\O$,
and one can easily find admissible couples
with any prescribed winding number.
%we can define two sets $E_{a'},E_{-b'}\subset \S^2$ 
%such that the boundary of $E_{a'}$ (resp. $E_{-b'}$) is parametrized by $a'$ (resp. $-b'$) with positive orientation.
%The conditions $E_{a'}\subset E_{-b'}$, $E_{-b'}\subset E_{a'}$, $E_{a'}\cap E_{-b'}=\emptyset$ 
%and $E_{a'}\cup E_{-b'}=\S^2$ should characterize the connected components of $\O$. 
If $n=4$ the linking number in $\S^3$ of the curves $a',-b'$
is constant on connected components of $\O$, and one can find admissible couples
with any prescribed linking number. If $n>4$
the assertion follows from the fact the every knot is trivial in $\S^n$.
\end{remark}

\section{Nonuniqueness of smooth extremal surfaces}\label{S:4}

In the following statement, we say that a surface $M\subset \R^{1+n}$ is {\em locally cylindrical}
if, for every $t_0\in \R$, there exists an open interval $I\subset \R$ such that $t_0\in I$ and
$M \cap (I\times \R^2)$ is a cylinder in $I\times \R^n$, i.e. it can be written in the form 
\eqref{formofpsi}.

\begin{proposition}\label{prop:nonuniq}
If $n\ge 3$ there exist two
distinct globally $C^\infty$ timelike extremal surfaces $M^1, M^2$ in $\R^{1+n}$,
both locally cylindrical, such that $M^1$ and $M^2$
coincide when $t\in [0,\delta]$ for some $\delta>0$, 
in the sense that
\begin{equation}\label{nonuniq0}
\left\{( t,x)\in M^1 : t \in [0,\delta] \right\}
=
\left\{( t,x)\in M^2 : \in [0,\delta] \right\}.
\end{equation}
%for some $\delta >0$.
\end{proposition}

The surface $M^2$ that we construct below
has the  property that
it is locally cylindrical but not globally cylindrical.

In general, in geometric evolution problems, self-intersections can give
rise to nonuniqueness.
The proposition shows that, even if we require smoothness
and impose the ``locally cylindrical" topological constraint,
one can still take advantage of self-intersections to
generate examples of nonuniqueness.

\begin{proof}
For $i=1,\ldots,3$ let  $a_i, b_i$ be distinct $C^\infty$ maps $\R\to \R^n$, 
periodic with period $1$, such that $|a_i'|=|b_i'|=1$ and 
such that
\begin{equation}
a_i(x) = b_j(x) = (x, 0,\ldots, 0)\qquad \mbox{ for all $i,j$ and all $x\in [-\delta, \delta]$ for some $\delta< \dfrac 12$\,.}
\label{coincide}\end{equation}
Assume  in addition that
\begin{equation}
\{ (s, \sigma) \in \R\times \R: a_i'(s)+ b_j'(\sigma)=0 \mbox{ for some }i,j \} = \emptyset\,.
\label{nonuniq1}\end{equation}
This says that no $b_j'$ ever passes through 
any point that is antipodal to any point on any  $a_i'$.
It follows easily from  Lemma \ref{lem:convexhull} that this can be accomplished.
%, since condition \eqref{coincide}  is clearly not an obstacle.

Now for any permutation $\pi:\{1,2,3\} \to \{1,2,3 \}$, let 
$(a^\pi,b^\pi)$ and  be periodic curves $\R\to \R^n$ with period $3$, 
defined by
\[
(a^\pi, b^\pi)(x) = (a_{\pi(i)}(x), b_{\pi(i)}(x))  \qquad \mbox{ for }x\in [i-1, i]\mod 3.
\]
Next, define $\gamma^\pi(t,x) := \frac 12 (a^\pi(x+t) + b^\pi(x-t))$,

Letting $id$ denote the identity permutation, we claim that for every $\pi$,
\begin{equation}
\mbox{ $\gamma^\pi(t, \cdot)$ and $\gamma^{id}(t, \cdot)$ parametrize the same curve for 
$0\le t \le \delta$}.
\label{samecurve}\end{equation}
Indeed, if $t\ge 0$, we have
\begin{equation}
\gamma^\pi(t,x) 
=
\begin{cases}
\frac 12(a_{\pi(i)}(x+t) + b_{\pi(i)}(x-t)) &\mbox{ if }i-1 \le x-t \le x+t \le i \mod 3\\
\frac 12(a_{\pi(i+1)}(x+t) + b_{\pi(i)}(x-t)) &\mbox{ if }i-1 \le x-t \le i \le  x+t \mod 3\,,
\end{cases}
\label{gammapi}\end{equation}
where addition of indices is understood mod 3. If $0 \le t \le \delta$,
it follows from this and  \eqref{coincide} that
\[
\gamma^\pi(t,x) = (x-i , 0 \ldots, 0)\ \mbox{ if  }i-1 \le x-t \le i \le x+t,
\]
for {\em every} permutation $\pi$.
Then one can see by inspection of  \eqref{gammapi} that \eqref{samecurve} holds 
(in fact it also holds for $t\in [-\delta, 0]$, by essentially the same argument).
Next, note that \eqref{gammapi} implies that
\[
\mbox{$\gamma^\pi(\frac 12, x) = 
\frac 12 (a_{\pi(i+1)}(x+\frac 12) + b_{\pi(i)}(x-\frac 12)) \quad \mbox{ if }
i - \frac 12 \le x \le i+ \frac 12 \mod 3$}
\]
and from this one can see that in general $\gamma^{id}(\frac 12, \cdot) \ne
\gamma^{\pi}(\frac 12, \cdot)$ if for example $\pi$ is an odd permutation.

Finally, define
$\psi^\pi(t,x) = (t, \gamma^\pi(t,x))$.
Let $M^1$ be the surface parametrized by $\psi^{id}$, and
let $M^2$ be the surface that agrees with $M^1$ when $t\le \delta$,
and for $t\ge \delta/2$ is parametrized by
$\psi^\pi(t,x)$ for some odd permutation $\pi$. This definition makes sense in view of \eqref{samecurve}.
These surfaces  have all the stated properties.
In particular, it follows from \eqref{nonuniq1}
and Theorem 4.1  that $M^1, M^2$ are both 
smoothly immersed and locally cylindrical.

(Indeed, note that since $\gamma^\pi$ and $\gamma^{id}$ as constructed above are both
periodic with period $3$ in the $t$ variable, we are free to switch back and forth at will between
$\gamma^\pi$ and $\gamma^{id}$ every $3$ units of $t$.)
\end{proof}

\begin{remark}\rm
When $n=2$, a similar argument yields two functions $\gamma^{id}, \gamma^\pi$
of the form \eqref{eqrepr}, \eqref{vincoloab} that parametrize the same curve for $|t|\le \delta$,
but not for all $t$. These functions $\gamma^{id}, \gamma^\pi$ fail to be global timelike
immersions, see Theorem \ref{thm:sing} below, but it is presumably possible to 
arrange that the breakdown of uniqueness (for the image manifolds) occurs before
the breakdown of regularity. 
\end{remark}

\begin{remark}\rm
In the proof of  Proposition \ref{prop:nonuniq}, if $a_1=a_2=a_3$ and $b_1, b_2, b_3$ are distinct, then one can see from
\eqref{gammapi} that $\gamma^{id}(t, \cdot)$ and $\gamma^\pi(t,\cdot)$ parametrize
the same curve for every $t$.
This shows that the different admissible pairs can generate the same 
extremal surface.
\end{remark}

\begin{remark}\rm We now provide an example of nonuniqueness of (weakly) extremal surfaces, due to 
the appearance of singularities in the evolution. 
Let $\gamma_1(t,x)=\frac{1}{2}(a_1(x+t)+a_1(x-t))$
and $\gamma_2(t,x)=\frac{1}{2}(a_2(x+t)+a_2(x-t))$ be orthogonal parametrization of two different 
global extremal cylinders $M_1$ and $M_2$, 
with $a_1, a_2$ arclength parametrizations of the boundaries of two
distinct uniformly convex, centrally symmetric planar sets, both 
periodic with period $E_0$. Symmetry implies that $a_i(x+E_0/2) = -a_i(x)$
for all $x$ and $i=1,2$, and thus 
\[
\gamma_i( E_0/4, x) = \frac 12\Big(a_i(x+ E_0/4) + a_i(x-E_0/4) \Big)= 0\qquad \mbox { for }x\in \R, i\in \{1,2\}.
\]
In other words,  $\gamma_1,\gamma_2$  both have
an extinction singularity at the origin at time $\bar t :=E_0/4$.

Note that the time derivatives at time $\bar t$ of $\gamma_1$ and $\gamma_2$ 
are respectively given by  $a'_1(x + \bar t)$  and $a'_2(x +\bar t)$.

Define now $\gamma(t,x)=\gamma_1(t,s(x))$  for $t<\bar t$ and
$\gamma(t,x)=\gamma_2(t,x)$ for $t\ge \bar t$, where $s(x)$ is a
reparametrization of $[0,E_0]$ such that
$$a'_1(s(x) + \bar t)=a'_2(x + \bar t)$$
for all $x\in [0,E_0]$, i.e.  $$s(x)= - \bar t +(a'_1)^{-1}\circ a'_2(x+\bar t).$$
It follows that the derivatives of $\gamma$ are continuous at $(x,\bar t)$ for any $x\in\R$, and hence $\gamma$ may be suitably extended to a $C^1$
{\it nonorthogonal} parametrization of a global (weakly) extremal cylinder $M$ that agrees respectively with $M_1$ and $M_2$ on disjoint time intervals.

%Notice finally that a $C^1$ parametrization $\gamma$ in the orthogonal gauge  in a neighborhood of an extinction singularity is necessarily unique.
\end{remark}

\section{dimension of the singular set}

Given $\gamma(t,x) = ( a(x+t)+b(x-t))/2$, with $a,b$ satisfying \eqref{vincoloab},
and $\psi(t,x) = (t,\gamma(t,x))$, we now prove some upper bounds on the size of the singular sets $Sing$ and 
$Sing^*$ associated to the cylinder $M={\rm Image}(\psi)$.
%We define 
%\[
%Sing := \{ \psi(t,x) : \mbox{rank}(\nabla \psi)(t,x) < 2 \}  \ = \ \{ \psi(t,x) : \gamma_x(t,x) = 0\}.
%\]
%Then $M$ is timelike and regularly immersed in an open neighborhood of every  point of $M\setminus Sing$, and at every point of $Sing$,  $M$ fails to be timelike
%and the orthogonal coordinate system degenerates.
%A stricter notion of singular set is
%\[
%Sing^* := \{ p\in Sing : \lim_{q\in M, q\to p} \tau(q)\mbox{ does not exist} \},
%\]
%where $\tau(\cdot)$ is the (spatial) tangent
%%\[
%%Reg := M \setminus Sing.
%%\]
%%We think of $Reg$ as the set of timelike, regularly immersed points of $M$. 
%%At points in $Sing$, $M$ certainly fails to be timelike, but it could otherwise be rather smooth.
%\[
%\tau(p) =  \frac{\gamma_x}{|\gamma_x|}\circ \psi^{-1}(p)
%\]
%defined wherever it makes sense, which is at points $p\in M\setminus Sing$ where
%the set $\{\frac{\gamma_x}{|\gamma_x|}(t,x) :
%\psi(t,x)=p\}$ consists of exactly one element. %It is possible that there is someother better notion of a strongly singular set in this context. 
%
%We note that the definition of $Sing^*$ has the drawback that it depends on
%the parametrization of $M$.
%We next give a more  detailed description of singularities that must occur in the case $n=2$. 
%

%We use the following terminology: If $\gamma$ is a periodic map $\R\to \R^n$, and if  $\hat \gamma$ is a monotone reparametrization such that $|\hat \gamma_x| = 1$ a.e., then we say that $\gamma$ is  singular if $\hat \gamma_x$ cannot be extended to a continuous map $\R\to \S^{n-1}$. 

In the following theorem, which is one of the main results of this paper, ``dim" always
means (Euclidean) Hausdorff dimension\footnote{It is arguably slightly unnatural to characterize a singular set 
in Minkowski space by the Euclidean Hausdorff dimension, but note that this quantity
is invariant with respect to Lorentz transformations. }. 
%We also use the notational convention $C^{k,0} = C^k$.

\begin{theorem} Assume that $a,b\in C^{k}(\R, \R^n)$, with $k\in\mathbb N$, with $(a,b)\in X$. Then,
 
\smallskip

{\bf 1}. $\mathcal H^{1+ \frac 1{k}}(Sing) = 0$ and $\mbox{dim}(Sing^*)\le
\mbox{dim}(Sing) \le 1+ \frac 1{k}$.

\smallskip

{\bf 2}. It can
%If $n\ge 3$ then for every $\theta\in [0, 1+ \frac 1{k}]$, it can
happen that $\mbox{dim}(Sing^*) = 1+\frac 1k $.

%\smallskip
%{\bf 3}. If $n=2$ then for every $\theta\in  \{ 0 \} \cup[1, 1+ \frac 1{k}]$, it can happen that 
%$\mbox{dim}(Sing^*) = \theta$.

\smallskip

{\bf 3}. When $n=2$ and $(a,b)\in X_{\rm per}$, (at least) one of the following properties holds:
\begin{itemize}
\item[-]  there exists $t_0$ such that $\gamma_x(t_0, x)=0$ for all $x$, 
%either there exists $\alpha\in [0,E_0]$ such that  $a'(s) +b'(s+\alpha)=0$ for all $s\in [0,E_0]$,
\item[-] 
$Sing^*$ is at least one-dimensional, and the set
\[
\{ t\in \R : \exists x\in \R \mbox{ such that }\psi(t,x)\in Sing^* \}
\]
contains an open interval.
\end{itemize}
\label{thm:sing}\end{theorem}

\begin{remark}\label{ps}\rm
In fact it is clear from the proof that conclusion {\bf 1} holds for any surface that
is given locally as the image of a $C^k$ map, including
any surface that can be written locally in the form \eqref{eqrepr}
with $a,b\in C^k$. In particular, this applies to
noncompact surfaces, as well as local cylinders of the type appearing in Proposition
\ref{prop:nonuniq}.

Also, we prove conclusion {\bf 2} for global cylinders, the most restrictive (topological) class
of functions considered in this paper,  so it follows that it holds
for other classes of surfaces  --- noncompact, locally cylindrical --- as well.
\end{remark}

Note also that Remark \ref{localcyl} applies to conclusion {\bf 3}.

%, where as usual $a,b$ satisfy \eqref{vincoloab}.
It is natural to wonder whether the results we prove here for
the weak solutions given by the explicit formula \eqref{eqrepr}
still hold in a larger class of weak solutions, and 
also whether any analogous
results hold for higher-dimensional extremal surfaces.

%\begin{theorem}\label{thm:sing} Assume that $a,b\in C^1(\R;\R^2)$ are $E_0$-periodic functions satisfying \eqref{vincoloab}, and define $\gamma$ by \eqref{eqrepr} as usual. Then  at least one of the following holds:
%\begin{itemize}
%\item[-]  there exists some $t_0$ such that $\gamma_x(t_0, x)=0$ for all $x$, or
%\item[-] 
%The set $Sing^*$ is at least one-dimensional, and indeed
%\[
%\{ t\in \R : \exists x\in \R \mbox{ such that }\psi(t,x)\in Sing^* \}
%\]
%contains an open interval.
%\end{itemize}
%\end{theorem}

Conclusion {\bf 3} is a  refinement of a result from \cite{NT}.
Our proof gives more details
than \cite{NT} concerning the situation described in \eqref{eq:notsing}, since we found this point not completely straightforward.
%The proof follows some arguments from \cite{NT}, where the theorem is essentially proved,  although our argument gives a little more information with slightly weaker hypotheses, and a more detailed consideration of the bad case \eqref{eq:notsing}.
The proof of {\bf 3} shows that, if  for instance $a$ is a nonconvex curve in $\R^2$ and $b(x) = - a(x+E_0/2)$, 
then {\em both} the  alternatives of conclusion {\bf 3} hold.

The rest of this section is devoted to the proof of Theorem \ref{thm:sing}.

\begin{proof}[Proof of {\bf 1}]
The estimate $\mathcal H^{1+\frac 1k}(Sing) = 0$
follows directly from a refined version of Sard's Theorem, see Federer \cite[3.4.3]{federer}.
%, and the other conclusions are immediate consequences.
\end{proof}

\begin{remark} \rm
In \cite{BN} one can find
a version of Sard's Theorem more refined than the one
cited above, which gives a necessary and sufficient condition
for a set $A\subset\R$ to be the set of critical values of some
function in $C^{k,\alpha}(\R,\R^n)$.
If $a,b\in C^{k,\alpha}$ for $\alpha\in(0,1)$ and $k$ is a positive integer,
these result implies that \[
%\mathcal H^{\frac 1{k+\alpha}}\Big(  \{ \psi(s,x) : \gamma_x(s,x) = 0\}  \Big)= 
\mathcal H^{\frac 1{k+\alpha}}\Big( Sing \cap\big( \{ s\}\times \R^n \big) \Big) = 0
\qquad\mbox{ for every }s \in\R.
% \{ (t,x)\in S : t= s  \Big)= 0
\]
It is reasonable to conjecture, and may be even easy to prove,
that under these hypotheses one has $\mathcal H^{1+\frac 1{k+\alpha}}(Sing) = 0$,
but this does not immediately follow from \cite{BN}.
However, 
straightforward modifications of the proof of {\bf 2} below show that
it can happen that $\mbox{dim}(Sing^*) = 1+ \frac 1{k+\alpha}$, when $a,b\in C^{k,\alpha}$.
\label{rem:Ckalpha}
\end{remark}

\smallskip

\begin{proof}[Proof of {\bf 2}]
It is enough to construct an example when $n=2$. In order to do it,
we will need the following result:

\begin{lemma}\label{lem:Sard}
For every positive integer $k$, there exists $f\in C^k([0,1];\R)$
such that
\begin{equation}
\mbox{dim}(f( \Sigma)) = \frac 1k,
\qquad\mbox{ where } \ 
\Sigma:= 
\big\{x\in [0,1] : f'\mbox{ changes sign near $x$}   \big\} .
\label{reducetof}\end{equation}
%and in addition $f'$ changes sign near every critical point of $f$.
\end{lemma}

The proof, which we defer to the end of this section, is a small  modification of a classical 
argument used by Federer  to prove the sharpness of his refined
version of Sard's Theorem which we cited in the proof of {\bf 1} above. 

We may assume that the function $f$ from Lemma \ref{lem:Sard} satisfies
$|f'|\le \frac 12$, since multiplying a function by a constant does not change the
dimension of the associated set $\Sigma$. 

We define $g\in C^k([0,1])$ such that $g' =  (1-f'^2)^{1/2}$
for $x\in [0,1]$, and we fix
two periodic maps  $a,b\in C^k(\R;\R^2)$, with the same period $E_0>3$, 
parametrized by arclength and such that
\[
a(x) = (f(x),g(x))\mbox{ for }x\in [0,1], \qquad b(x) = (0, -x) \ \mbox{ for }x\in [-1,2].
\]
%(Note that $|a'| = |b'| = 1$ in $[0,1]$, as required.)
Then
\[
\gamma_x(t,x) = \frac 12(f'(x+t), g'(x+t)-1 ) \quad\mbox{ if }x+t\in [0,1],\ |t|\le \frac 12,
\]
and since $|g'-1| \le C f'^2$,  it follows that $\frac{\gamma_x}{|\gamma_x|}$ is discontinuous
at all points $(t,x)$ such that $|t|\le \frac 12$ and $x+t\in \Sigma$.
We then deduce that
\begin{align*}
Sing^*
\ &\supset \ 
\left\{(t,\frac{ f(x+t)}2, \frac{ g(x+t) - (x-t) } 2)  : \ x+t\in\Sigma,  \ |t|\le  \frac12  \right\} \\
&\ =  \ 
\left\{(0,  \frac{ f(s)}2, \frac{ g(s)  - s } 2)    + (t, 0, t) :  \ s \in \Sigma, \ |t|\le\frac 12 \right\} .
\end{align*}
If we let
\[
A_0 := \left \{ \frac 12(f(x), g(x)+x) :  x\in \Sigma \right\},
\]
then $\mbox{dim}(Sing^*) \ge \mbox{dim}(A_0) +1$, since $Sing^*\subset\R^{1+2}$ contains a
copy of $A_0\subset \R^2$ translated along a  line segment.
Moreover, since $\frac 12 f(\Sigma)$ is the projection of $A_0$ on the $x$-axis,
we conclude from \eqref{reducetof} that
\[
\mbox{dim}(Sing^*) \ge \
1+ \mbox{dim}(A_0) \ge \
1+ \mbox{dim}\Big( \frac 12 f(\Sigma) \Big) =  1+\frac 1k\,.
\]

We remark that, although the map $\gamma$ constructed above is singular
for $t=0$, one can easily modify the construction to arrange that $\gamma$ is
regularly immersed at $t=0$ and develops singularities as described above at a later time. 
Indeed, $\gamma$ is a regular immersion at $t=0$  if $a' = e^{i\alpha}, b' = -e^{i\beta}$, and $\alpha < \beta < \alpha+2\pi$,
and this condition can be achieved, while essentially preserving the above construction, by choosing $E_0$ large enough, taking a certain amount
of care in how $\alpha$ is defined in $[1,E_0]$ and $\beta$ in $[2, E_0-1]$,
and then replacing $a(\cdot)$ by $a(\cdot - E_0/2)$.
\end{proof}

\smallskip

\begin{proof}[Proof of {\bf 3}]
Extend $a', b'$ to $E_0$-periodic maps from $\R$ to $\S^1$, and let
$\alpha,\beta :\R\to \R$ be two continuous functions such that
\begin{equation}
a'(x) = e^{i\alpha(x)}, \qquad  -b' =  e^{i\beta(x)}.
\label{alphabeta}\end{equation}
As in the proof of Corollary \ref{coruno},
%about the images of $a'$ and $b'$ imply that 
$a'$ and $b'$ satisfy \eqref{eqab}, which implies that 
the images of $a'$ and $-b'$ are
closed arcs with intersection of positive length.
In particular, by adding $2\pi k$ to $\alpha$ for an appropriate integer $k$, we can assume that
the set Image$(\alpha) \cap \mbox{Image}(\beta)$ contains an interval of positive length. It then follows that the function
\[
F(t,x) :=   \alpha(x+t) - \beta(x-t) 
\]
takes both positive and negative values. For example, to find a point where
$F>0$, choose $s,\sigma\in \R$ such that $\alpha(s)> \beta(\sigma)$,
and let $(t,x) = ( \frac 12 (s-\sigma), \frac 12 (s+\sigma))$.

%\smallskip

%Note also that the periodicity of $\alpha,\beta$ implies that $F(t,x+E_0) = F(t,x)$.

%\smallskip

We shall consider 2 cases:

\noindent{\it Case 1}. For every $t_0$, the function $x\mapsto F(t_0,x)$ does not change sign.

Then, since $F$ assumes both positive and negative values, 
there must be some $t_0$ such that $F(t_0,x)=0$
for all $x$. It follows that $\gamma_x(t_0, x) = 0$
for all $x$.

\noindent{\it Case 2}. There exists some $t_0$ such that $x\mapsto F(t_0,x)$ changes sign.

Then by continuity $x\mapsto F(t,x)$ changes sign for all $t$ in a neighborhood of $t_0$.

Fix such a $\bar t$, and let $S$ be a connected component of the set $\{ x : F(\bar t, x) = 0\}$
such that $F$ assumes both positive and negative values in every neighborhood of 
$S$.
As observed in \cite{NT}, it follows from \eqref{alphabeta} that
the unit tangent $\tau = \frac {\gamma_x}{|\gamma_x|}$ is given by
\begin{equation}
\frac{\gamma_x}{|\gamma_x|}(\bar t,x) = \operatorname{sign}\left( \sin\left( \frac 12 F(\bar t,x) \right)\right)\ 
i\, e^{\frac i 2\, G(\bar t,x)},
\qquad G(\bar t,x):= \alpha(x+\bar t)+\beta(x-\bar t)
\label{angles}\end{equation}
wherever $\gamma_x\ne 0$ (for simplicity of notation we identify $\R^2$ with $\mathbb C$).
Therefore, if $S$ consists of a single point $(\bar t,x)$, then  
%$\lim_{q\in M, q\to \gamma(t,x)}\tau(q)$
$\lim_{y\to x}\frac{\gamma_x}{|\gamma_x|}(\bar t,y)$
does not exist, which implies that $\psi(\bar t,x)\in Sing^*$.
% the unit tangent vector $\tau := \gamma_x/|\gamma_x|$ is discontinuous at that point.

Suppose now that $S$ is an interval, say $S = [s_0,s_1]$.
Then
$\gamma_x(\bar t,x) = 0$ for all $x\in S$, so that $x\mapsto \gamma(\bar t,x)$ is 
constant for $x\in S$. 
It follows that 
$\psi(\bar t, s_0)\in Sing^*$ unless
%$\gamma(\cdot, t)$ is singular unless
\begin{equation}
\label{eq:notsing}
\tau(\bar t,s_0^-) 
%:= \lim_{x\nearrow s_0}\tau(t,x)
=
\tau(\bar t,s_1^+) , %:= \lim_{x\searrow s_1}\tau(t,x)
\end{equation}
where $\tau(\bar t,s_0^-)  := \lim_{x\nearrow s_0}\tau(\bar t,x)$
and $\tau(\bar t,s_1^+)  := \lim_{x\searrow s_1}\tau(\bar t,x)$.
(Condition \eqref{eq:notsing} includes the assertion that
both limits exist.)
Recalling that $F$ changes sign near $S$, we deduce from \eqref{angles} that  
\eqref{eq:notsing} can only occur if
$\operatorname{sign}(F(\bar t, s_0^-)) = - \operatorname{sign}(F(\bar t, s_1^+))$
and 
\begin{equation}
\frac 12({G(\bar t,s_1) - G(\bar t,s_0)}) = \pi \mod 2\pi.
\label{eq:notsing2}\end{equation}
Assume this holds  and let 
\[
x_0 := \min\left\{ x\in S : |\alpha(x+\bar t) - \alpha(s_0 +\bar t)| = \frac \pi 3\right\}>s_0\,.
\]
Since
$0 = F(\bar t,x) = \alpha(x+\bar t) - \beta(x-\bar t)$ for all $x\in S$, we get
\begin{eqnarray*}
F(\bar t+\eps, x_0- |\eps|) &=& 
% = \begin{cases}
\alpha(x_0+\bar t+\eps-|\eps| ) - \beta(x_0 -\bar  t -\eps -|\eps|) 
\\
&=& \alpha(x_0 + \bar t +\eps-|\eps|) - \alpha(x_0 + \bar t-\eps -|\eps| )
\ne 0
%\alpha(x_0 +t - 2|\eps| ) - \beta(x_0 ) = 
%- [\alpha(x_0 - 2|\eps| +t) - \alpha(x_0 ) ]&\mbox{ if }\eps<0.
%\end{cases}
\end{eqnarray*}
for all $\eps$ such that $x_0-2|\eps| > s_0$. 
The last inequality follows from the fact that 
$|\alpha(x_0+\bar t)  -  \alpha(s_0 +\bar  t)| = \pi/3$, while
$|\alpha(x_0 + \bar t - 2|\eps|) - \alpha(s_0+\bar t)| < \pi/3$, by the choice of $x_0$.

{}From the equality $F(\bar t,x) = 0$, for these $\eps$ it also follows that  
\[
F(\bar t - \eps, x_0 -|\eps|) = - F(\bar t+\eps, x_0-|\eps|). %\qquad \mbox{ if }\ 0 < \eps < \frac 12(x_0-s_0).
\]
In particular, the function $\eps \mapsto F(\bar t+\eps, x_0-|\eps|)$ changes sign at $\eps=0$.

Fix now $y_0<s_0$ such that $F(\bar t,y_0)\ne 0$ and
\[
|\alpha(y+ \bar t) - \alpha(s_0+\bar t)| < \frac \pi 3\quad
\mbox{ for all }y\in [y_0, s_0].
\]
For all $\eps$ sufficiently small, $F( t+\eps, y_0+|\eps|)$ has the same
sign as $F(\bar t, y_0)$.
Thus, for all  $\eps$ in an interval of the form $(-\delta,0)$ or
$(0,\delta)$, the function
$x\mapsto F(\bar t+ \eps, x)$ must change sign between $y_0+|\eps|$ and
$x_0-|\eps|$.
For such $\eps$, the interval $(y_0+|\eps|, x_0-|\eps|)$ must contain
a connected component $\hat S = [\hat s_0, \hat s_1]$ of $\{ x : F(\bar t+\eps,
x) = 0\}$
such that $F$ assumes both positive and negative values in every neighborhood of $\hat S$.
Moreover, our choice of
$y_0$ and $x_0$ guarantees that, possibly reducing $\delta$, we have
\begin{align*}
\pi
\ & >  \
|\alpha(\hat s_1 + \bar t+\eps) - \alpha(s_0+\bar t)| + |\alpha(s_0+\bar t) - \alpha(\hat
s_0+\bar  t+\eps)|
\\
& > \
|\alpha(\hat s_1 + \bar t+\eps)  - \alpha(\hat s_0+ \bar t+\eps)|\\
& = \
 \frac 12\,|G(\bar t+\eps, \hat s_1) - G(\bar t+\eps, \hat s_0)|
\end{align*}
%
%\[
% \pi > |\alpha(\hat s_1 + t+\eps) - \alpha(\hat s_0 +t+\eps)|  = 
% \frac 12|G(t+\eps, \hat s_1) - G(t+\eps, \hat s_0)|
%\]
where the last equality follows from the fact that $F(\bar t+\eps, s) = 0$
in $[\hat s_0, \hat s_1]$. Thus \eqref{eq:notsing2} cannot hold, and hence
for all
$\eps$ in the interval that we have found, $\psi(\bar t+\eps, \hat s_0)\in Sing^*$, thus completing the
proof of {\bf 3}.
%For all sufficiently small $\eps$, $F( t+\eps, y_0)$ has the same sign
%as $F(t, y_0)$.
%Thus for all  $\eps$ in an interval of the form $(-\delta,0)$ or $(0,\delta)$, the function
%$x\mapsto F(t+ \eps, x)$ must change sign between $y_0$ and $x_0-|\eps|$.
%For such $\eps$, the interval $(y_0, x_0-|\eps|)$ must contain
%a connected component $\hat S = [\hat s_0, \hat s_1]$ of $\{ x : F(t+\eps, x) = 0\}$
%such that $F$ assumes both positive and values in every neighborhood of $\hat S$,
%and our choices of 
%$y_0$ and $x_0$ guarantee that, after taking $\delta$  smaller if necessary, 
% \[
% \pi > |\alpha(t+\eps, \hat s_0) - \alpha(t+\eps, \hat s_1)|  = \frac 12|G(\eps, \hat s_1) - G(\eps, \hat s_0)|
% \]
%Thus \eqref{eq:notsing2} cannot hold, and hence for all
%$\eps$ in the interval that we have found, $\psi(t+\eps, \hat s_0)\in Sing$, completing the
%proof of {\bf 3}.
\end{proof}

Finally we give a proof of Lemma \ref{lem:Sard}.

\begin{proof}[Proof of Lemma \ref{lem:Sard}]
We divide the proof into three steps.\\
{\bf Step 1.} We first recall Federer's proof that for any $k\in \mathbb N$
and $\mu\in (0,\frac 1k)$, there exists $g\in C^k([0,1])$ such that
%$f'$ changes sign near every critical point of $f$
\begin{equation}
\mathcal H^{\mu}\Big(
\big\{ g(x) : g'(x) = 0   \big\} \Big) 
>0.
\label{reducetof.wk}\end{equation}
For $\sigma\in (0,1)$ we will write $C_\sigma$ to denote the ``middle $\sigma$" Cantor-type
set, so that 
\[
C_\sigma = \cap_{\ell=1}^\infty\cup_{i \in \{ 0,1\}^\ell} C_\sigma(i)
\] 
where, for every $\ell$ and every $i\in \{0,1\}^\ell$,
$C_\sigma(i)$ is a closed interval of length 
$(\frac {1-\sigma}2)^\ell$, and 
%\[
%C_{\sigma}(i_1,\ldots, i_\ell) \setminus R_\sigma(i_1,\ldots, i_\ell)
%= 
$C_\sigma (i_1,\ldots, i_\ell, 0)$, 
$C_\sigma (i_1,\ldots, i_\ell, 1)$ are obtained by removing from
$C_\sigma (i_1,\ldots, i_\ell)$ a centered open interval of length
$\sigma(\frac {1-\sigma}2)^\ell$. As usual we start with
$C_\sigma(0) = [0, \frac 12(1-\sigma)]$ and $C_\sigma(1) = [\frac 12(1+\sigma)2, 1]$,
and we
label the intervals so that
$C_\sigma (i_1,\ldots, i_\ell, 0)$ 
lies to the left of 
$C_\sigma (i_1,\ldots, i_\ell, 1)$.

Fix now $\nu$ and $\delta>0$ such that $(k+\delta)\mu = \nu <1$, and let $\alpha, \beta\in (0,1)$ satisfy
\[
\left(\frac{1-\alpha}2\right)^\mu = \left(\frac {1-\beta}2\right)^\nu = \frac 12.
\]
These numbers are chosen so that $C_\alpha$ and $C_\beta$ have dimension $\mu$ and $\nu$
respectively, and $\mathcal H^\mu(C_\alpha), \mathcal H^\nu(C_\beta)>0$.

Notice that there is a natural map $g_0:C_\beta\to C_\alpha$,
characterized by 
\[
g_0( C_\beta \cap C_\beta(i))  = 
C_\alpha \cap C_\alpha(i)\qquad\quad\mbox{ for every $\ell\in \mathbb N$ and }i\in \{0,1\}^\ell.
\]
As Federer noted in \cite[3.4.4]{federer}, $g_0$ extends to a $C^k$ map $g: [0,1]\to [0,1]$
by a routine application of the Whitney extension Theorem.
The point is that, given $x,y\in C_\beta$, we can fix $\ell\in \mathbb N$ and $i\in \{0,1\}^\ell$
such that $x,y\in C_\beta(i)$, but $x$ and $y$ belong to different subintervals of $C_\beta(i)$.
Then $g_0(x)$ and $g_0(y)$ both belong to $C_\alpha(i)$, and from this information one can easily check that
\[
|x-y| \ge \beta(\frac {1-\beta}2)^{\ell}, \qquad
|g_0(x) - g_0(y)| \le (\frac {1-\alpha}2)^{ \ell} =  (\frac {1-\beta}2)^{(k+\delta) \ell}.
\]
As a result we get
$$|g_0(x)-g_0(y)|\le C (\frac{1-\beta}2)^{\delta\ell}|x-y|^k  = o( |x-y|^k),$$ 
hence Whitney's Theorem yields a $C^k$ extension of $g$ as required.

It is also clear that $g'=0$ in $C_\beta$, so that
every point of $C_\alpha$ is a critical value of $g$, and  \eqref{reducetof.wk} holds.

{\bf Step 2}. 
We modify the above construction to produce a function $f\in C^k([0,1])$
such that, 
for a fixed $\mu< \frac 1k$, we have
\begin{equation}
\mathcal H^\mu (f( \Sigma)) \ge 0
\qquad\mbox{ where } \ 
\Sigma:= 
\big\{x\in [0,1] : f'\mbox{ changes sign near $x$}   \big\} .
\label{eq:fmu}\end{equation}
To do this, we fix $\alpha,\beta$ as above and define $f_0:C_\beta\to C_\alpha$,
characterized by
\[
f_0( C_\alpha \cap C_\alpha(i))  = 
C_\beta \cap C_\beta(i^*)\qquad\quad\mbox{ for every $\ell\in \mathbb N$ and }i\in \{0,1\}^\ell,
\]
where for every $k$ and every $i\in \{0,1\}^k$, we define $i^*$ 
by
\begin{eqnarray*}
i^*_j = i_j \quad \mbox{ if $j$ is odd},
\qquad \qquad
i^*_j = i_j +1 \mod 2 \quad \mbox{ if $j$ is even}.
\end{eqnarray*}
Then, as in the classical argument described above, $f_0$
extends to a $C^k$ function $f:[0,1]\to \R$. In addition, we have the inclusion 
$C_\beta \subset \Sigma$, 
since every interval $C_\alpha(i)$ for $i$ odd contains points
such that $x<y$ and $f(x)< f(y)$, whereas for $i$ even $C_\alpha(i)$
contains points such that $x<y$ and $f(x)> f(y)$.

{\bf Step 3}. For every $m>k$, let $f_m$ be a function satisfying
\eqref{eq:fmu} with $\mu = \frac 1k - \frac 1m$, and extend $f_m$ so that
$f_m' = 0$ on $\R\setminus [0,1]$.
We define
\[
f(x) := \sum_{k=1}^\infty h_m f_m\left( 2^{m+1} (x - 2^{-m}) \right)
\]
for a sequence $(h_m)$ decreasing to zero fast enough 
so that the series converges in $C^k$. Then $f\in C^k$ and satisfies \eqref{reducetof}.
\end{proof}

%Note also that there are integers $d_a, d_b$ such that 
%\begin{equation}
%\alpha(x+E_0) = \alpha(x) + 2\pi d_a,\qquad
%\beta(x+E_0) = \beta(x) + 2\pi d_b.
%\label{eq:periodicmod}\end{equation}
%Indeed, $d_a$ and $d_b$ are just the degrees of $a', b':\R/E_0\mathbb Z\cong S^1\to S^1$.Since $(\Gamma, v)$ is admissible,  $\gamma_t(0,x) = \frac 12|a'(x) - b'(x)| < 1$ for all $x$, which implies that $\alpha(x) \ne \beta(x) \mod 2\pi$ for all $x\in \R$, and hence that $d_a= d_b =: d$.

%%%%%%%%%%%%%%%%%%%%%%%%%%%%%%%%%%%%%%%%%%%%%%%%%%%%%%%%%%%%%%%%%%%%%%%%
\bibliographystyle{plain}
%%%%%%%%%%%%%%%%%%%%%%%%%%%%%%%%%%%%%%%%%%%%%%%%%%%%%%%%%%%%%%%%%%%%%%%%

\end{document}